\documentclass[12pt]{article}
\usepackage{amsmath, amssymb,latexsym}
\usepackage{color}
\title{Weyl symbols and boundedness of Toeplitz operators}
\author{Lewis A. Coburn \\\small Department of Mathematics \\\small SUNY at Buffalo
\\\small Buffalo \\\small NY 14260, USA\\\small lcoburn@buffalo.edu \and
Michael Hitrik \\\small Department of Mathematics \\\small University of California
\\\small Los Angeles \\\small CA 90095-1555, USA\\\small hitrik@math.ucla.edu \and
Johannes Sj\"ostrand \\\small IMB, Universit\'e de Bourgogne \\\small 9, Av. A. Savary, BP 47870
\\\small FR-21078 Dijon, France\\\small and UMR 5584 CNRS
\\\small johannes.sjostrand@u-bourgogne.fr \and
Francis White \\\small Department of Mathematics \\\small University of California
\\\small Los Angeles \\\small CA 90095-1555, USA\\\small fwhite@math.ucla.edu}
\date{}

\def\wrtext#1{\relax\ifmmode{\leavevmode\hbox{#1}}\else{#1}\fi}
\def\abs#1{\left|#1\right|}
\def\begeq{\begin{equation}}
\def\endeq{\end{equation}}

\def\iint{\int\hskip -2mm\int}

\def\part#1{\frac{\partial}{\partial #1}}

\def\norm#1{||\,#1\,||}

\newcommand{\real}{\mbox{\bf R}}
\newcommand{\comp}{\mbox{\bf C}}

\renewcommand{\exp}{\mbox{\rm exp\,}}

\newtheorem{dref}{Definition}[section]
\newtheorem{lemma}[dref]{Lemma}
\newtheorem{theo}[dref]{Theorem}
\newtheorem{prop}[dref]{Proposition}

\newenvironment{proof}{\vspace{.3cm}\noindent{{\em Proof:}}}{\hfill$\Box$}

\begin{document}

\maketitle

\vspace*{1cm}
\noindent
{\bf Abstract}: We study Toeplitz operators on the Bargmann space, with Toeplitz symbols that are exponentials of inhomogeneous quadratic polynomials. It is shown that the boundedness of such operators is implied by the boundedness of the corresponding Weyl symbols.


\vskip 2mm
\noindent
{\bf Mathematics Subject Classifi\-ca\-tion 2000}: 32U05, 32W25, 35S30, 47B35

%

\section{Introduction and statement of results}
\setcounter{equation}{0}

In the recent work~\cite{CHiSj}, the authors have established some basic links between the theory of Toeplitz operators acting on exponentially weighted spaces of entire holomorphic functions and Fourier integral operators (FIOs) in the complex domain. The point of view of complex FIOs was used in~\cite{CHiSj} to show that the boundedness of a certain class of Toeplitz operators is implied by the boundedness of their Weyl symbols, in agreement with a general conjecture made in~\cite{BC94}. The purpose of this note is to obtain a slight, but perhaps natural, extension of this result, by taking a closer look at the arguments of~\cite{CHiSj}. In a special case, we show that the boundedness of the Weyl symbols is also a necessary condition for the boundedness of the corresponding Toeplitz operators. We shall now proceed to describe the assumptions and state the main results.

\medskip
\noindent
Let $\Phi_0$ be a strictly plurisubharmonic quadratic form on $\comp^n$ and let us set
\begeq
\label{eq1.1}
\Lambda_{\Phi_0} = \left\{\left(x,\frac{2}{i}\frac{\partial \Phi_0}{\partial x}(x)\right), \, x\in \comp^n\right\} \subset \comp^{2n}.
\endeq
The real $2n$-dimensional linear subspace $\Lambda_{\Phi_0}$ is I-Lagrangian and R-symplectic, in the sense that the restriction of the complex symplectic form on $\comp^{2n}$ to $\Lambda_{\Phi_0}$ is real and non-degenerate. In particular, $\Lambda_{\Phi_0}$ is maximally totally real.

\medskip
\noindent
Let us introduce the Bargmann space
\begeq
\label{eq1.2}
H_{\Phi_0}(\comp^n) = L^2(\comp^n, e^{-2\Phi_0} L(dx)) \cap {\rm Hol}(\comp^n),
\endeq
and the orthogonal projection
\begeq
\label{eq1.3}
\Pi_{\Phi_0}: L^2(\comp^n,e^{-2\Phi_0} L(dx)) \rightarrow H_{\Phi_0}(\comp^n).
\endeq
Here $L(dx)$ is the Lebesgue measure on $\comp^n$. In this note we shall be concerned with the boundedness properties of Toeplitz operators of the form
\begeq
\label{eq1.4}
{\rm Top}(e^Q) = \Pi_{\Phi_0} \circ e^{Q} \circ \Pi_{\Phi_0}: H_{\Phi_0}(\comp^n) \rightarrow H_{\Phi_0}(\comp^n),
\endeq
where $Q$ is an inhomogeneous quadratic polynomial on $\comp^n$ with complex coefficients. The following is the main result of this work.

\begin{theo}
\label{theo_main}
Let $\Phi_0$ be a strictly plurisubharmonic quadratic form on $\comp^n$ and let $Q$ be a quadratic polynomial on $\comp^n$ with the principal part $q$. Assume that
\begeq
\label{eq1.5}
{\rm Re}\, q(x) < \Phi_{\rm herm}(x) := (1/2)\left(\Phi_0(x) + \Phi_0(ix)\right),\quad x \neq 0
\endeq
and
\begeq
\label{eq1.6}
{\rm det}\, \partial_x \partial_{\overline{x}} \left(2\Phi_0 - q\right) \neq 0.
\endeq
Let $a\in C^{\infty}(\Lambda_{\Phi_0})$ be the Weyl symbol of the operator ${\rm Top}(e^{Q})$ and assume that
$a \in L^{\infty}(\Lambda_{\Phi_0})$. Then the Toeplitz operator
$$
{\rm Top}(e^{Q}): H_{\Phi_0}(\comp^n) \rightarrow H_{\Phi_0}(\comp^n)
$$
is bounded.
\end{theo}

\medskip
\noindent
{\it Remark}. In the homogeneous case, when $Q$ is a quadratic form, Theorem \ref{theo_main} was established in~\cite{CHiSj}. In the general inhomogeneous case considered here, Theorem \ref{theo_main} provides further evidence for the conjecture of~\cite{BC94},~\cite{LC}, stating that a Toeplitz operator is bounded on $H_{\Phi_0}(\comp^n)$ if and only if the corresponding Weyl symbol is bounded on $\Lambda_{\Phi_0}$.

\medskip
\noindent
The plan of this note is as follows. In Section \ref{sect_two}, we carry out the principal step in the proof of Theorem \ref{theo_main} by characterizing boundedness properties of operators given as Weyl quantizations of symbols of the form $e^{iP(x,\xi)}$, where $P$ is a holomorphic inhomogeneous quadratic polynomial on $\comp^{2n}$. The homogeneous case was discussed in~\cite{CHiSj}, and the only additional idea required here consists of performing a factorization of a suitable complex affine canonical transformation associated to the Weyl quantization above. The proof of Theorem \ref{theo_main} is then completed in Section \ref{sect_three} by passing from the Toeplitz symbol to the Weyl one, along the lines of~\cite{CHiSj}. Section \ref{sect_four} is devoted to the discussion of an explicit family of metaplectic Toeplitz operators in a quadratic Bargmann space, where we also verify that the sufficient condition for the boundedness of the Toeplitz operator given in Theorem \ref{theo_main} is in fact necessary, in agreement with the conjecture of~\cite{BC94},~\cite{LC}.

\medskip
\noindent
{\bf Acknowledgments}. J.S. acknowledges the support from the 2018 Stefan Bergman award.

\section{From bounded Weyl symbols to bounded Weyl quantizations}
\label{sect_two}
\setcounter{equation}{0}
\noindent

\medskip
\noindent
Let $F(x,\xi)$ be a holomorphic quadratic form on $\comp^{2n}$, let $\ell(x,\xi)$ be a complex linear function on $\comp^{2n}$, and let us consider formally the Weyl quantization of a symbol of the form
\begeq
\label{eq2.2}
a(x,\xi) = \exp(i(F(x,\xi) + \ell(x,\xi))).
\endeq
We have
\begin{equation}
\label{eq2.3}
Au(x)= a^w(x,D_x)u(x) = \frac{1}{(2\pi )^n}\iint e^{i((x-y)\cdot \theta +F((x+y)/2,\theta) + \ell((x+y)/2,\theta) )}u(y)dy d\theta .
\end{equation}
Following~\cite{CHiSj}, we shall view $A$ as a Fourier integral operator in the complex domain.  The holomorphic quadratic polynomial
\begeq
\label{eq2.4}
\Phi(x,y,\theta) = (x-y)\cdot \theta +F((x+y)/2,\theta) + \ell((x+y)/2,\theta)
\endeq
is a non-degenerate phase function in the sense of H\"ormander and defines a canonical relation
\begeq
\label{eq2.5}
\kappa: (y,-\partial_y \Phi(x,y,\theta)) \mapsto (x,\partial_x \Phi(x,y,\theta)), \quad \partial_{\theta} \Phi(x,y,\theta) = 0.
\endeq
Writing $\eta = -\partial_y \Phi(x,y,\theta)$ and $\xi = \partial_x \Phi(x,y,\theta)$ we see that $\kappa$ is given by $(y,\eta) \mapsto (x,\xi)$, where
\begin{equation}
\label{eq2.6}
\begin{split}
x&=\frac{x+y}{2}-\frac{1}{2}F'_\xi \left(\frac{x+y}{2},\theta\right) - \frac{1}{2} \ell'_{\xi},\\
y&=\frac{x+y}{2}+\frac{1}{2}F'_\xi \left(\frac{x+y}{2},\theta\right) + \frac{1}{2} \ell'_{\xi},\\
\xi& =\theta +\frac{1}{2}F'_x \left(\frac{x+y}{2},\theta\right) + \frac{1}{2} \ell'_x,\\
\eta &=\theta -\frac{1}{2}F'_x \left(\frac{x+y}{2},\theta\right) - \frac{1}{2} \ell'_x.
\end{split}
\end{equation}
Here $\ell'_x$, $\ell'_{\xi} \in \comp^n$ are constant. The graph of $\kappa$ is parametrized by the midpoint coordinate
$$
\rho = \left(\frac{x+y}{2},\theta\right) \in \comp^{2n},
$$
and we may rewrite (\ref{eq2.6}) in the form
\begeq
\label{eq2.7}
\kappa: \rho + \frac{1}{2} H_{F + \ell}(\rho) \mapsto \rho - \frac{1}{2} H_{F + \ell}(\rho).
\endeq
Here $H_{F + \ell}(\rho) = (F'_{\xi}(\rho) + \ell'_{\xi}, -F'_x(\rho) - \ell'_x)$ is the Hamilton vector field of the holomorphic function $F + \ell$ at $\rho$. Recalling as in~\cite{CHiSj} that the Hamilton vector field of $F$ is given by $H_F(\rho) = {\cal F}\rho$, where
$$
{\cal F}= \begin{pmatrix}F''_{\xi x} &F''_{\xi \xi }\\
-F''_{xx} &-F''_{x\xi }\end{pmatrix}
$$
is the fundamental matrix of $F$, we see that (\ref{eq2.7}) takes the form
\begeq
\label{eq2.8}
\kappa: \left(1 + \frac{1}{2} {\cal F}\right)\rho + \frac{1}{2} H_{\ell} \mapsto \left(1 - \frac{1}{2} {\cal F}\right)\rho - \frac{1}{2} H_{\ell}.
\endeq
In what follows we shall assume that $\pm 2 \notin {\rm Spec}({\cal F})$, so that the canonical relation
\begeq
\label{eq2.9}
\kappa_F: \left(1 + \frac{1}{2} {\cal F}\right)\rho \mapsto \left(1 - \frac{1}{2} {\cal F}\right)\rho
\endeq
is a canonical transformation. We have
$$
\kappa = \exp\left(-\frac{1}{2} H_{\ell}\right) \circ \kappa_F \circ \exp\left(-\frac{1}{2} H_{\ell}\right).
$$
More explicitly, it follows from (\ref{eq2.8}) that $\kappa$ is a complex affine canonical transformation given by
\begeq
\label{eq2.10}
\kappa: \rho \mapsto \kappa_F(\rho) - \frac{1}{2}\kappa_F(H_{\ell}) - \frac{1}{2}H_{\ell}.
\endeq
In view of Jacobi's theorem, the right hand side of (\ref{eq2.10}) is given by
$$
\kappa_F(\rho) - \frac{1}{2} H_{\ell \circ \kappa_F^{-1} + \ell},
$$
and we conclude that the map $\kappa$ admits the following factorization
\begeq
\label{eq2.11}
\kappa = \kappa_{\ell} \circ \kappa_F,
\endeq
where $\kappa_F$ is given in (\ref{eq2.9}) and $\kappa_{\ell}$ is a complex phase space translation given by
\begeq
\label{eq2.12}
\kappa_{\ell}(\rho) = \rho - \frac{1}{2} H_{\ell \circ \kappa_F^{-1} + \ell}.
\endeq

\bigskip
\noindent
Let $\Phi_0$ be a strictly plurisubharmonic quadratic form on $\comp^n$ and let us recall the I-Lagrangian R-symplectic linear manifold
$\Lambda_{\Phi_0}$ defined in (\ref{eq1.1}). The following is the main result of this section.

\begin{theo}
\label{theo_sect2}
Let $F$ be a holomorphic quadratic form on $\comp^{2n}$ such that the fundamental matrix of $F$ does not have the eigenvalues $\pm 2$, and let $\ell$ be a complex linear form on $\comp^{2n}$. Let $\Phi_0$ be a strictly plurisubharmonic quadratic form on $\comp^n$. Let
$$
a(x,\xi) = \exp(i(F(x,\xi) + \ell(x,\xi)))
$$
and assume that $a \in L^{\infty}(\Lambda_{\Phi_0})$. Then the operator
$$
a^w(x,D_x): H_{\Phi_0}(\comp^n) \rightarrow H_{\Phi_0}(\comp^n)
$$
is bounded.
\end{theo}

\medskip
\noindent
When proving Theorem \ref{theo_sect2}, we shall rely on some results of~\cite{CHiSj}, and it will also be convenient to use the factorization (\ref{eq2.11}). Our starting point is the following observation.

\begin{lemma}
\label{lemma_sect2}
Let $m(x,\xi)$ be a complex linear form on $\comp^{2n}$ and let us consider the complex canonical transformation $\exp(H_m)(\rho) = \rho + H_m$, $\rho \in \comp^{2n}$, (a complex phase space translation). Let $\Phi$ be a strictly plurisubharmonic quadratic form on $\comp^n$. Then we have
$$
\exp(H_m)\left(\Lambda_{\Phi}\right) = \Lambda_{\Psi},
$$
where $\Psi$ is a strictly plurisubharmonic quadratic polynomial on $\comp^n$ given by
\begeq
\label{eq2.14}
\Psi(x) = \Phi(x) + {\rm Im}\, \left(m\left(x,\frac{2}{i}\frac{\partial \Phi}{\partial x}(x)\right)\right),\quad x \in \comp^n.
\endeq
\end{lemma}

\begin{proof}
While (\ref{eq2.14}) can be established by a straightforward computation, here we would like to indicate a more general approach, illustrating the point of view of evolution equations associated to the operator $m^w(x,D)$. See also~\cite{Sj83},~\cite{HiPrVi}. Let us consider the real Hamilton-Jacobi equation
\begeq
\label{eq2.14.1}
\frac{\partial \Psi}{\partial t}(x,t) - {\rm Im}\, m\left(x,\frac{2}{i}\frac{\partial \Psi}{\partial x}(x,t)\right) = 0, \quad \Psi(x,0) = \Phi(x),
\endeq
for $x\in \comp^n$, $t\in \real$, $t\geq 0$. Associated to the function $\Psi(x,t)$ is the manifold
$$
L_{\Psi} = \left\{\left(t, \frac{\partial \Psi}{\partial t}, x, \frac{2}{i}\frac{\partial \Psi}{\partial x}\right)\right\} \subset \real^2_{t,\tau} \times \comp^{2n}_{x,\xi},
$$
which is Lagrangian with respect to the real symplectic form
\begeq
\label{eq2.14.2}
d\tau \wedge dt - {\rm Im}\, \sigma,
\endeq
where
$$
\sigma = \sum_{j=1}^n d\xi_j \wedge dx_j
$$
is the complex symplectic (2,0)--form on $\comp^{2n}_{x,\xi}$. The equation (\ref{eq2.14.1}) tells us that
$$
\left(\tau - {\rm Im}\, m\right)|_{L_{\Psi}} = 0,
$$
and therefore the Hamilton vector field of the function $\tau - {\rm Im}\, m$, computed with respect to the real symplectic form (\ref{eq2.14.2}), is tangent to $L_{\Psi}$. Using the general relation
$$
\widehat{H_m} = H_{-{\rm Im}\, m}^{-{\rm Im}\, \sigma},
$$
valid for any $m(x,\xi)$ holomorphic, where $\widehat{H_m} = H_m + \overline{H_m}$ is the real vector field naturally associated to the holomorphic vector field $H_m$, see~\cite{Sj82}, we conclude that the vector field
$$
\partial_t + H_{-{\rm Im}\, m}^{-{\rm Im}\, \sigma} = \partial_t + \widehat{H_m}
$$
is tangent to $L_{\Psi}$. Identifying $\widehat{H_m}$ and $H_m$, we get
$$
\Lambda_{\Psi(\cdot,t)} = \exp(t H_m)\left(\Lambda_{\Phi}\right).
$$
It is now easy to obtain (\ref{eq2.14}) and to this end, we claim that the unique solution of the equation (\ref{eq2.14.1}) is given by
\begeq
\label{eq2.14.3}
\Psi(x,t) = \Phi(x) + t {\rm Im}\, \left(m\left(x,\frac{2}{i}\frac{\partial \Phi}{\partial x}(x)\right)\right) + C_t,
\endeq
where $C_t$ depends on $t$ only. When verifying the claim, let us write $-{\rm Im}\, m = p$ and choose real linear coordinates on $\comp^n$ so that $(x,\frac{2}{i}\partial_x \Psi(x,t))$ corresponds to $(x,\partial_x \Psi(x,t))$ in the usual $\real^{2n}$--sense. Then (\ref{eq2.14.1}) becomes
\begeq
\label{eq2.14.4}
\frac{\partial \Psi}{\partial t}(x,t) + p\left(x,\frac{\partial \Psi}{\partial x}(x,t)\right) = 0, \quad \Psi(x,0) = \Phi(x).
\endeq
Here $p(x,\xi)$ is real linear on $\real^{2n}_x \times \real^{2n}_{\xi}$, $p(x,\xi) = p'_x \cdot x + p'_{\xi} \cdot \xi$ and $\Phi(x)$ is a real quadratic form on $\real^{2n}$,
$$
\Phi(x) = \frac{1}{2} A_0 x \cdot x, \quad x\in \real^{2n},
$$
with $A_0$ real symmetric. With
$$
\Psi(x,t) = \frac{1}{2} A_t x\cdot x + B_t\cdot x + C_t,
$$
the equation (\ref{eq2.14.4}) becomes
$$
\frac{1}{2} \partial_t A_t x\cdot x + \partial_t B_t \cdot x + \partial_t C_t + p'_x \cdot x + p'_{\xi} \cdot\left(A_t x + B_t\right) = 0,
$$
and we immediately get the unique solution
$$
\Psi(x,t) = \frac{1}{2} A_0 x \cdot x - t \left(p'_x \cdot x + p'_{\xi} \cdot A_0 x\right) + C_t = \Phi(x) - t p\left(x,\partial_x \Phi(x)\right) + C_t,
$$
where
$$
C_t = \frac{t^2}{2} p'_{\xi} \cdot \left(p'_x + A_0 p'_{\xi}\right)
$$
depends on $t$ only. This shows (\ref{eq2.14.3}) and completes the proof.
\end{proof}

\medskip
\noindent
{\it Remark}. Associated to the canonical transformation $\exp(H_m)$ is the Fourier integral operator $e^{-im^w(x,D)}$, and from~\cite{Sj96} we may recall the explicit description
$$
e^{-im^w(x,D)} = e^{-\frac{i}{2} m'_x\cdot x} \circ \tau_{m'_{\xi}} \circ e^{-\frac{i}{2} m'_x\cdot x},
$$
where $\tau_s$ is the operator of translation by $s\in \comp^n$, $(\tau_s u)(x) = u(x-s)$. We may then verify by an explicit computation that the operator $e^{-im^w(x,D)}$ is bounded,
$$
e^{-im^w(x,D)}: H_{\Phi}(\comp^n) \rightarrow H_{\Psi}(\comp^n),
$$
where $\Psi$ is given by (\ref{eq2.14}). Here the weighted spaces of holomorphic functions $H_{\Phi}(\comp^n)$, $H_{\Psi}(\comp^n)$ are defined analogously to (\ref{eq1.2}).

\bigskip
\noindent
Let $a$ be of the form (\ref{eq2.2}) and let us notice that $a\in L^{\infty}(\Lambda_{\Phi_0})$ precisely when
\begeq
\label{eq2.15}
{\rm Im}\, F|_{\Lambda_{\Phi_0}}\geq 0
\endeq
and
\begeq
\label{eq2.16}
\rho \in \Lambda_{\Phi_0},\,\, {\rm Im}\, F(\rho) = 0 \Longrightarrow {\rm Im}\, \ell (\rho) = 0.
\endeq
It follows from (\ref{eq2.15}) and Proposition B.1 in~\cite{CHiSj} that the canonical transformation $\kappa_F$ in (\ref{eq2.9}) is positive relative to $\Lambda_{\Phi_0}$, and applying Theorem 1.1 of~\cite{CHiSj}, we get
\begeq
\label{eq2.17}
\kappa_F(\Lambda_{\Phi_0}) = \Lambda_{\Phi},
\endeq
where $\Phi$ is a strictly plurisubharmonic quadratic form such that $\Phi \leq \Phi_0$. We need to obtain an explicit description of the (clean) intersection $\Lambda_{\Phi}\cap \Lambda_{\Phi_0}$.

\begin{prop}
\label{prop_sect2}
We have
\begeq
\label{eq2.18}
\Lambda_{\Phi} \cap \Lambda_{\Phi_0} = \left\{\left(1-\frac{1}{2} {\cal F}\right)\rho;\,\, \rho \in \Lambda_{\Phi_0},\,\, {\rm Im}\, F(\rho) = 0\right\} \subset \comp^{2n}.
\endeq
\end{prop}
\begin{proof}
It will be convenient to obtain a reduction to the case when $\Lambda_{\Phi_0}$ is replaced by the real phase space $\real^{2n}$. To this end, let $T: L^2(\real^n) \rightarrow H_{\Phi_0}(\comp^n)$ be a unitary metaplectic Fourier integral operator with the associated complex linear canonical transformation $\kappa_T$ such that
$$
\kappa_T(\real^{2n}) = \Lambda_{\Phi_0}.
$$
We have
$$
\Lambda_{\Phi} \cap \Lambda_{\Phi_0} = \kappa_F(\Lambda_{\Phi_0}) \cap \Lambda_{\Phi_0} = \kappa_T \left(\kappa_G(\real^{2n}) \cap \real^{2n}\right),
$$
where it follows from (\ref{eq2.9}) that
$$
\kappa_G: \left(1 + \frac{1}{2} {\cal G}\right)\rho \mapsto \left(1 - \frac{1}{2} {\cal G}\right)\rho,
$$
and ${\cal G}$ is the fundamental matrix of the quadratic form $G = F \circ \kappa_T$. We have ${\rm Im}\, G|_{{\bf R}^{2n}}\geq 0$ and therefore
$$
\left(1 \pm \frac{1}{2} {\cal G}\right)\rho \in \real^{2n}
$$
precisely when $\rho \in \real^{2n}$, ${\rm Im}\, G(\rho) = 0$. It follows that
$$
\kappa_G(\real^{2n}) \cap \real^{2n} = \left\{\left(1 - \frac{1}{2} {\cal G}\right)\rho; \, \rho \in \real^{2n},\,{\rm Im}\, G(\rho) = 0\right\},
$$
and we obtain (\ref{eq2.18}).
\end{proof}

\medskip
\noindent
In what follows we shall use the notation
\begeq
\label{eq2.19}
L = \left\{\left(1-\frac{1}{2} {\cal F}\right)\rho;\,\, \rho \in \Lambda_{\Phi_0},\,\, {\rm Im}\, F(\rho) = 0\right\}.
\endeq
Letting $\pi_x: \comp^{2n} \ni (x,\xi) \mapsto x\in \comp^n$ be the projection map, we notice that
\begeq
\label{eq2.20}
\left\{x\in \comp^n; \Phi_0(x) = \Phi(x)\right\} = \pi_x \left(\Lambda_{\Phi}\cap \Lambda_{\Phi_0}\right) = \pi_x L,
\endeq
and the quadratic form $\Phi_0 - \Phi\geq 0$ satisfies
\begeq
\label{eq2.21}
\Phi_0(x) - \Phi(x) \simeq {\rm dist}(x,\pi_x L)^2,\quad x\in \comp^n.
\endeq

\medskip
\noindent
We shall now consider the I-Lagrangian R-symplectic affine plane $\kappa(\Lambda_{\Phi_0})$, where $\kappa$ is given by (\ref{eq2.8}). It follows from (\ref{eq2.11}), (\ref{eq2.12}), Lemma \ref{lemma_sect2}, and (\ref{eq2.17}) that
\begeq
\label{eq2.22}
\kappa(\Lambda_{\Phi_0}) = \Lambda_{\Psi},
\endeq
where $\Psi$ is a strictly plurisubharmonic quadratic polynomial on $\comp^n$ given by
\begeq
\label{eq2.23}
\Psi(x) = \Phi(x) + {\rm Im}\,\left(m\left(x,\frac{2}{i}\frac{\partial \Phi}{\partial x}(x)\right)\right),
\endeq
where
\begeq
\label{eq2.24}
m = -\frac{1}{2}\left(\ell \circ \kappa_F^{-1} + \ell\right).
\endeq

\medskip
\noindent
We claim that the quadratic polynomial $\Phi_0 - \Psi$ vanishes along the real linear subspace $\pi_x L \subset \comp^n$ and to this end, it suffices to check that the linear form $m$ is real along $L \subset \Lambda_{\Phi}$. It follows from (\ref{eq2.9}), (\ref{eq2.16}), and (\ref{eq2.24}) that when $\rho \in \Lambda_{\Phi_0}$, ${\rm Im}\, F(\rho) = 0$, we have
$$
m\left(\left(1-\frac{1}{2}{\cal F}\right)\rho\right) = - \ell (\rho)
$$
is real. We have therefore verified the claim and using also (\ref{eq2.21}) and (\ref{eq2.23}) we conclude that the inhomogeneous quadratic polynomial $\Phi_0 - \Psi$ is bounded below on $\comp^n$. The general theory, see~\cite{Sj82},~\cite{CaGrHiSj}, together with (\ref{eq2.22}),  allows us to conclude that the operator
$$
a^w(x,D_x): H_{\Phi_0}(\comp^n) \rightarrow H_{\Psi}(\comp^n)
$$
is bounded, and this completes the proof of Theorem \ref{theo_sect2}.

\section{Toeplitz operators and proof of Theorem 1.1}
\label{sect_three}
\setcounter{equation}{0}
The purpose of this section is to apply the results of Section \ref{sect_two} to the study of boundedness properties of Toeplitz operators in the Bargmann space, establishing Theorem \ref{theo_main}.

\medskip
\noindent
Let $\Phi_0$ be a strictly plurisubharmonic quadratic form on $\comp^n$ and let $Q$ be a quadratic polynomial with complex coefficients on $\comp^n$, with the principal part $q$. Assume that the condition (\ref{eq1.5}) holds. Arguing as in~\cite[Section 4]{CHiSj}, we then see that when equipped with the natural domain
\begeq
\label{eq3.1}
{\cal D}({\rm Top}(e^Q)) = \left\{u\in H_{\Phi_0}(\comp^n); e^{Q} u \in L^2(\comp^n, e^{-2\Phi_0}L(dx))\right\},
\endeq
the Toeplitz operator
\begeq
\label{eq3.2}
{\rm Top}(e^{Q}) = \Pi_{\Phi_0} \circ e^{Q} \circ \Pi_{\Phi_0}: H_{\Phi_0}(\comp^n) \rightarrow H_{\Phi_0}(\comp^n)
\endeq
is densely defined.

\medskip
\noindent
Recalling the integral representation for the orthogonal projection $\Pi_{\Phi_0}$ and following~\cite{CHiSj}, we may write for $u\in {\cal D}({\rm Top}(e^Q))$,
\begeq
\label{eq3.3}
{\rm Top}(e^{Q})u(x) = C \int\!\!\!\int_{\Gamma} e^{2(\Psi_0(x,\theta) - \Psi_0(y,\theta)) + Q(y,\theta)} u(y)\, dy\, d\theta,
\endeq
where $\Psi_0$ is the polarization of $\Phi_0$ and $\Gamma$ is the contour in $\comp^{2n}$, given by $\theta = \overline{y}$. Using the assumption (\ref{eq1.6}), we conclude as in~\cite{CHiSj} that the operator ${\rm Top}(e^Q)$ can be viewed as a metaplectic Fourier integral operator associated to a complex affine canonical transformation: $\comp^{2n} \rightarrow \comp^{2n}$.

\medskip
\noindent
It is now easy to complete the proof of Theorem \ref{theo_main}. Let us write, following~\cite{Sj96},~\cite{CHiSj},
\begeq
\label{eq3.4}
{\rm Top}(e^Q) = a^w(x,D_x),
\endeq
where $a\in C^{\infty}(\Lambda_{\Phi_0})$ is the Weyl symbol of the Toeplitz operator ${\rm Top}(e^{Q})$, given by
\begeq
\label{eq3.5}
a\left(x,\xi\right)  = \left(\exp\left(\frac{1}{4} \left(\Phi''_{0,x\overline{x}}\right)^{-1} \partial_x \cdot \partial_{\overline{x}}\right)e^Q\right)(x), \quad (x,\xi) \in \Lambda_{\Phi_0}.
\endeq
In~\cite{CHiSj}, we have seen that
\begeq
\label{eq3.6}
a(x,\xi) = C_{\Phi_0} \int_{{\bf C}^n} \exp(-4 \Phi_{{\rm herm}}(x-y)) e^{Q(y)}\, L(dy),\quad C_{\Phi_0}\neq 0,
\endeq
where the integral converges thanks to (\ref{eq1.5}). An application of the method of exact stationary phase allows us therefore to conclude that
\begeq
\label{eq3.7}
a(x,\xi) = C \exp(i\left(F(x,\xi) + \ell(x,\xi)\right)), \quad (x,\xi) \in \Lambda_{\Phi_0},
\endeq
for some $C\neq 0$, where $F$ is a holomorphic quadratic form on $\comp^{2n}$ and $\ell$ is a complex linear function on $\comp^{2n}$. Theorem \ref{theo_main} follows therefore from Theorem \ref{theo_sect2}.

\section{Example: boundedness of a metaplectic Toep\-litz operator}
\label{sect_four}
\setcounter{equation}{0}
In the beginning of this section we shall illustrate Theorem \ref{theo_main} by applying it in the case when
\begeq
\label{eq4.1}
\Phi_0(x) = \frac{\abs{x}^2}{4},
\endeq
and
\begeq
\label{eq4.2}
Q(x) = \lambda \abs{x}^2 + \frac{1}{2} \overline{c} \cdot x - \frac{1}{2} d \cdot \overline{x}.
\endeq
Here $c,d \in \comp^n$ and $\lambda \in \comp$ satisfies ${\rm Re}\, \lambda < 1/4$, so that the conditions (\ref{eq1.5}), (\ref{eq1.6}) are satisfied. It follows from (\ref{eq3.5}) that the Weyl symbol $a$ of the operator ${\rm Top}(e^Q)$ is given by
\begeq
\label{eq4.3}
a\left(x,\frac{2}{i} \frac{\partial \Phi_0}{\partial x}(x)\right) = \left(\exp\left(\frac{1}{4} \Delta\right) e^Q\right)(x) = \frac{1}{\pi^n} \int_{{\bf C}^n} e^{-\abs{x-y}^2} e^{Q(y)}\, L(dy).
\endeq
Here $\Delta$ is the Laplacian on $\comp^n \simeq \real^{2n}$. The Gaussian integral in (\ref{eq4.3}) can be computed by the exact version of stationary phase and we get, after a straightforward computation,
\begeq
\label{eq4.4}
a\left(x,\frac{2}{i} \frac{\partial \Phi_0}{\partial x}(x)\right) = C\, \exp \left(\frac{1}{1-\lambda} \left(\lambda \abs{x}^2 + \frac{1}{2} \overline{c}\cdot x - \frac{1}{2} d \cdot \overline{x}\right)\right).
\endeq
Here $C\neq 0$ is a suitable constant depending on $\lambda$, $c$, $d$ only.

\medskip
\noindent
Using (\ref{eq4.4}), we may determine the explicit necessary and sufficient conditions for the boundedness of $a$ along $\Lambda_{\Phi_0}$. When doing so, it is convenient to introduce the parameter
\begeq
\label{eq4.5}
\gamma = \frac{1}{1-2\lambda},
\endeq
and to observe that
\begeq
\label{eq4.6}
{\rm Re}\, \left(\frac{\lambda}{1 - \lambda}\right) = \frac{1}{4 \abs{1-\lambda}^2} \left(1 - \frac{1}{\abs{\gamma}^2}\right).
\endeq
It follows, in particular, that if $\abs{\gamma} < 1$, then $a\in L^{\infty}(\Lambda_{\Phi_0})$ for all $c,d\in \comp^n$, and if $\abs{\gamma} > 1$, then $a$ is unbounded for all $c,d\in \comp^n$. In the ``boundary'' case when $\abs{\gamma} = 1$, we have $a\in L^{\infty}(\Lambda_{\Phi_0})$ precisely when
\begeq
\label{eq4.7}
{\rm Re}\, \left( \frac{\overline{c}}{1-\lambda} \cdot x - \frac{d}{1-\lambda}\cdot \overline{x}\right) = 0,\quad x\in \comp^n.
\endeq
Rewriting the condition (\ref{eq4.7}) in the form
$$
\left(\frac{\overline{c}}{1-\lambda} - \frac{\overline{d}}{\overline{(1-\lambda)}}\right)\cdot x + \left(\frac{c}{\overline{(1-\lambda)}}- \frac{d}{1-\lambda}\right)\cdot \overline{x} = 0, \quad x\in \comp^n,
$$
we conclude that $a\in L^{\infty}(\Lambda_{\Phi_0})$ precisely when
$$
c = \frac{\overline{1-\lambda}}{1-\lambda}\, d \Longleftrightarrow c = \gamma d.
$$

\medskip
\noindent
An application of Theorem \ref{theo_main} gives the following result.

\begin{prop}
\label{prop_example}
Let $\Phi_0(x) = \abs{x}^2/4$ and
$$
Q(x) = \lambda \abs{x}^2 + \frac{1}{2} \overline{c} \cdot x - \frac{1}{2} d \cdot \overline{x},
$$
with $c,d \in \comp^n$ and $\lambda \in \comp$, ${\rm Re}\, \lambda < 1/4$. Let us define $\gamma \in \comp$ as in {\rm (\ref{eq4.5})}. If
$\abs{\gamma} < 1$ then the operator
$$
{\rm Top}(e^Q): H_{\Phi_0}(\comp^n) \rightarrow H_{\Phi_0}(\comp^n)
$$
is bounded, for all $c,d \in \comp^n$. The same conclusion holds if $\abs{\gamma} = 1$ and $c = \gamma d$.
\end{prop}

\medskip
\noindent
We shall finish this section by demonstrating that, in the special case at hand, the condition $a\in L^{\infty}(\Lambda_{\Phi_0})$ is in fact also necessary for the boundedness of the Toeplitz operator ${\rm Top}(e^Q)$. When doing so, we shall study the action of ${\rm Top}(e^Q)$ on the normalized reproducing kernels for the Bargmann space $H_{\Phi_0}(\comp^n)$. To this end, let us first recall from~\cite{Sj96} that the orthogonal projection
$\Pi_{\Phi_0}: L^2(\comp^n, e^{-2\Phi_0}\, L(dx)) \rightarrow H_{\Phi_0}(\comp^n)$ is given by
\begeq
\label{eq4.8}
\Pi_{\Phi_0} u(x) = a_{\Phi_0} \int e^{2\Psi_0(x,\overline{y})} u(y) e^{-2\Phi_0(y)}\, L(dy), \quad a_{\Phi_0} > 0.
\endeq
Here
\begeq
\label{eq4.9}
\Psi_0(x,y) = \frac{1}{4} x \cdot y, \quad x,y\in \comp^n,
\endeq
is the polarization of $\Phi_0$. We have
\begeq
\label{eq4.10}
2{\rm Re}\, \Psi_0(x,\overline{y}) - \Phi_0(x) - \Phi_0(y) = -\Phi''_{0,\overline{x}x}(x-y)\cdot (\overline{x-y}) = -\frac{1}{4}\abs{x-y}^2.
\endeq

\medskip
\noindent
Let us set
\begeq
\label{eq4.11}
k_w(x) = (2\pi)^{-n/2} e^{2\Psi_0(x,\overline{w}) - \Phi_0(w)},\quad w\in \comp^n.
\endeq
Using (\ref{eq4.10}) we see that $k_w \in H_{\Phi_0}(\comp^n)$ with
\begeq
\label{eq4.12}
\norm{k_w}^2_{H_{\Phi_0}({\bf C}^n)} = \int \abs{k_w(x)}^2 e^{-2\Phi_0(x)}\, L(dx) = 1,\quad w\in \comp^n.
\endeq
We shall now consider the operator ${\rm Top}(e^Q)$ acting on $k_w$. To this end, it will be convenient to start by making the following observations. First, letting $q(x) = \lambda \abs{x}^2$ be the principal part of $Q$ in (\ref{eq4.2}), we obtain, in view of (\ref{eq4.8}) and the exact stationary phase,
\begeq
\label{eq4.13}
\left({\rm Top}(e^q) e^{2\Psi_0(\cdot, \overline{w})}\right)(x) = C_{\lambda} e^{2\Psi_0(x,\gamma \overline{w})}, \quad w\in \comp^n,
\endeq
where $C_{\lambda}$ is a constant depending on $\lambda$ only, and the parameter $\gamma$ has been defined in (\ref{eq4.5}). Next, let $h$ be entire holomorphic such that $\overline{h} e^{2\Psi_0(\cdot, \overline{w})} \in L^2(\comp^n,e^{-2\Phi_0}L(dx))$ for all $w\in \comp^n$. We then have
\begeq
\label{eq4.14}
\left({\rm Top}(\overline{h})e^{2\Psi_0(\cdot, \overline{w})}\right)(x) = \overline{h(w)} e^{2\Psi_0(x,\overline{w})}.
\endeq
Indeed, it suffices to observe that in view of (\ref{eq4.8}), the left hand side of (\ref{eq4.14}) is equal to
$$
\overline{\left({\rm Top}(h)e^{2\Psi_0(\cdot, \overline{x})}\right)(w)}.
$$
Finally, let $h$ be entire holomorphic  such that
$$
\overline{h} e^{2\Psi_0(\cdot,\overline{w})}, \overline{h} e^q e^{2\Psi_0(\cdot,\overline{w})} \in L^2(\comp^n,e^{-2\Phi_0}L(dx)),
$$
for all $w\in \comp^n$. Directly from the definitions we then see that
\begeq
\label{eq4.15}
{\rm Top}(\overline{h} e^q) = {\rm Top}(\overline{h}) {\rm Top}(e^q),
\endeq
when acting on the linear span of $\{e^{2\Psi_0(\cdot,\overline{w})},\, w\in \comp^n\} \subset H_{\Phi_0}(\comp^n)$.

\bigskip
\noindent
Using (\ref{eq4.2}), (\ref{eq4.11}), (\ref{eq4.13}), (\ref{eq4.14}), and (\ref{eq4.15}), we get
\begeq
\label{eq4.16}
\left({\rm Top}(e^Q) k_w\right)(x) = C_{\lambda} \exp(2\Psi_0(x, \gamma(\overline{w} + \overline{c})) - 2\Psi_0(d, \gamma(\overline{w} + \overline{c})) - \Phi_0(w)).
\endeq
Here, as above, $C_{\lambda} \neq 0$ is a constant which depends on $\lambda$ only. Taking the norm in $H_{\Phi_0}(\comp^n)$ and using (\ref{eq4.10}), we obtain
\begeq
\label{eq4.17}
\norm{{\rm Top}(e^Q) k_w}_{H_{\Phi_0}({\bf C}^n)} = C\, \exp(\Phi_0(\overline{\gamma}(w+c)) - \Phi_0(w) - 2{\rm Re}\, \Psi_0(d, \gamma \overline{w})).
\endeq
Here $C \neq 0$ is a constant depending on $\lambda$, $c,d$ only. It follows from (\ref{eq4.17}) that if $\abs{\gamma} > 1$, the operator ${\rm Top}(e^Q)$ is unbounded for all $c,d\in \comp^n$. If $\abs{\gamma} = 1$, we get with a new constant,
\begeq
\label{eq4.18}
\norm{{\rm Top}(e^Q) k_w}_{H_{\Phi_0}({\bf C}^n)} = C\, \exp(2{\rm Re}\, \Psi_0(w,\overline{c}) - 2{\rm Re}\, \Psi_0(w,\overline{\gamma}\overline{d})),
\endeq
and it follows that if $c \neq \gamma d$, the operator ${\rm Top}(e^Q)$ is unbounded.

\medskip
\noindent
The discussion above may be summarized in the following theorem.

\begin{theo}
Let $\Phi_0(x) = \abs{x}^2/4$ and $Q(x) = \lambda \abs{x}^2 + \frac{1}{2} \overline{c} \cdot x - \frac{1}{2} d \cdot \overline{x}$,
with $c,d \in \comp^n$ and $\lambda \in \comp$, ${\rm Re}\, \lambda < 1/4$. The Toeplitz operator
$$
{\rm Top}(e^Q): H_{\Phi_0}(\comp^n) \rightarrow H_{\Phi_0}(\comp^n)
$$
is bounded if and only if the Weyl symbol $a\in C^{\infty}(\Lambda_{\Phi_0})$ of ${\rm Top}(e^Q)$ satisfies $a\in L^{\infty}(\Lambda_{\Phi_0})$.
\end{theo}

\end{document}